\documentclass[12pt,a4paper]{article}
\usepackage{amsmath,amssymb}
\usepackage{amsfonts,amsthm,enumerate}
\newtheorem{theorem}{Theorem}
\newtheorem{lemma}[theorem]{Lemma}
\theoremstyle{definition}

\newcommand{\generate}[1]{\langle #1\rangle}

\title{\bf  The existence of free non-cyclic subgroups in weakly locally finite division rings}

\author { Bui Xuan Hai\footnote{Corresponding author, Faculty of Mathematics and Computer Science, University of Science, VNU-HCMC, 227 Nguyen Van Cu Str., Dist. 5, HCM-City, Vietnam,  e-mail: bxhai@hcmus.edu.vn} ~and  Nguyen Kim Ngoc\footnote{Faculty of Mathematics and Computer Science, University of Science, VNU-HCMC, 227 Nguyen Van Cu Str., Dist. 5, HCM-City, Vietnam,  e-mail: nkngoc1985@gmail.com}}

\begin{document}
\baselineskip=18pt
\maketitle
\def\Q{\mathbb{Q}}
\newcommand{\ts}[1]{\langle #1\rangle}
\newcommand{\Fp}{\mathbb{Z}_p}
\begin{abstract} In this paper we prove that every non-central subnormal subgroup of the multiplicative group of a weakly locally finite division ring contains free non-cyclic subgroups. 
\end{abstract}



\vspace*{0.5cm}

In \cite{tits}, J. Tits proved that every finitely generated subgroup of a matrix group over a field is either solvable-by-finite or contains non-cylic free group. In [2, p. 736], S. Bachmuth raised the question of whether Tits' theorem would be true for a matrix group over a division ring. In \cite{lich1}, Lichtman proved that Tits' theorem fails even for matrices of degree one, i.e. for $D^*={\rm GL}_1(D)$, where $D$ is a division ring. In this work \cite{lich1}, Lichtman remarked that the question of whether the multiplicative group of a division ring contains a non-cyclic free group remains without the answer. For the convenience, in \cite{Gon-Man}, Z. Goncalves and A. Mandel formulate this Lichtman's question as the following conjecture: 

\noindent
{\bf Conjecture 1.} {\em The multiplicative group of a non-commutatvie division ring contains a non-cyclic free subgroup.}

Moreover, they posed the following stronger conjecture:

\noindent
{\bf Conjecture 2.} {\em Any non-central subnormal subgroup of the multiplicative group of a non-commutatvie division ring contains a non-cyclic free subgroup.}

In \cite{Gon1}, Goncalves proved that Conjecture 1 is true for centrally finite division rings. In \cite{reich}, Reichstein and Vonessen showed that the multiplicative group $D^*$ of a division ring $D$ contains a  free subgroup on two generators if the center $F$ of $D$ is uncountable and there exists a non-central element $a\in D$ which is  algebraic over $F$ [13, Theorem 1]. Later, Chiba [3, Theorem 2] proved this result without the assumption of the existence of such an element $a\in D$.  In \cite{Gon-Man}, Goncalves and Mandel showed that Conjecture 2 is true in some particular cases. More exactly, they proved that a subnormal subgroup $G$ of $D^*$ contains a non-cyclic free subgroup if  $G$ contains some element $x\in D\setminus F$, which is algebraic over the center $F$ of $D$ such that either (a) ${\rm Gal}(F(x)/F)\neq 1$ or (b) $x^p\in F$, where $p=2$ or $p={\rm char}(F)>0$. The affirmative answer to Conjecture 2 was also obtained for centrally finite division rings by Goncalves in  \cite{Gon2}. Note that the affirmative answer to the question above would imply some known theorems like the commutativity theorems  of Kaplansky, Jacobson, Hua, Herstein, Stuth,... For more information, we refer to \cite{her2}. In this paper, we shall prove that Conjecture 2  is true for weakly locally finite division rings. The notion of weakly locally finite division rings was introduced firstly in \cite{hbd}. Recall that a division ring $D$ is called {\em weakly locally finite},  if for every finite subset $S$ of $D$, the division subring of $D$ generated by $S$ is centrally finite. It was proved in \cite{hbd} that every locally finite division ring is weakly locally finite and the converse is not true. In \cite{hbd}, it was given the example of weakly locally finite division ring which is not even algebraic over the center. Therefore, the class of weakly locally finite division rings is vast and our result is a broad generalization of Goncalves' result in \cite{Gon2}.  For more information about these division rings and their properties, we refer to \cite{hbd} and \cite{dbh}.

The following  more important Tits' result in \cite{tits} is often refered as Tits' Alternative and will be used in the proof of our Theorem 1 in the next: 
\\ \\
\noindent
{\bf Tits' Alternative.} \cite{tits} {\em  For a field $F$, the following statements hold:

(i) If ${\rm char}(F)=0$, then every subgroup of ${\rm GL}_n(F)$ either is solvable-by-finite or contains non-cyclic free subgroups.

(ii) If $H$ is a finitely generated subgroup of ${\rm GL}_n(F)$, then either $H$ is solvable-by-finite or $H$ contains  non-cyclic free subgroups.}

We need also one result from \cite{haz-mah-mot} and for the convenience, we now restate it. 
\\ \\
 \noindent
{\bf Lemma A.} \cite{haz-mah-mot} {\em Let $D$ be a centrally division ring. If $G$ is a solvable subgroup of $D^*$, then $G$ contains an abelian normal subgroup $A$ of finite index.}

Now we are ready to prove the following theorem, which can be considered as Tits' Alternative for weakly locally finite division rings.

\begin{theorem}\label{th:2.1}
Let $D$ be a weakly locally finite division ring with center $F$ and $G$ be a subgroup of $D^*$. Then, the following statements are equivalent:

(i) $G$ contains no non-cyclic free subgroups.

(ii) $G$ is locally solvable-by-finite.

(iii) $G$ is  locally abelian-by-finite.

(iv) Every finitely generated subgroup of $G$ satisfies a group identity. 
\end{theorem}

\begin{proof}

\noindent
$(i)\Longrightarrow (ii)$.  Suppose  (i) holds and consider a finite subset $S$ of $G$. Let $K$ and $H$ be respectively a division subring of $D$  and  a subgroup of $G$  generated by $S$. Since $D$ is  weakly locally finite, $K$ is centrally finite. If $n=[K: Z(K)]$, then $H$ can be viewed as a subgroup of ${\rm GL}_n(Z(K))$. Then, by Tits' Alternative, $H$ is solvable-by-finite; hence $G$ is locally solvable-by-finite. 

\noindent
$(ii)\Longrightarrow (iii)$. Using the same symbols  as above in (i), we see that there exists a normal solvable subgroup $A$ of $H$ such that $H/A$ is finite. Since $H$ is a subgroup of $K^*$, by Lemma A, there exists a normal abelian subgroup $B$ of $A$ such that $A/B$ is finite. So, if $C=\cap_{h\in H} B^h$, then $C$ is a normal subgroup of $H$ such that $H/C$ is finite. Thus, $H$ is abelian-by-finite, and (iii) holds.

\noindent
$(iii)\Longrightarrow (iv)$. Let $H$ be a finitely generated subgroup of $G$. Then, there exists a normal abelian subgroup $A$ of $H$ such that $[H: A]=k$. Then, $[x^k, y^k]$ is an identity for $H$. 

 \noindent
$(iv)\Longrightarrow (i)$. It is evident.
 \end{proof}

\begin{lemma}\label{lem:2.2} Let $D$ be a division ring algebraic over its center $F$. If $G$ is a (locally solvable)-by-(locally finite) subnormal subgroup of $D^*$ then $G\subseteq F$.
\end{lemma}

\begin{proof} By the assumption, there exists a locally solvable normal subgroup $S$ of $G$ such that $G/S$ is locally finite. So, $S$ is a locally solvable subnormal subgroup of $D^*$ and by [8, Theorem 2.4], $S\subseteq F$. Hence, $G/Z(G)$ is locally finite, so by [1, Lemma 3], $G'$ is locally finite. Thus, $G'$ is a torsion subnormal subgroup of $D^*$ and by [11, Theorem 8], $G'\subseteq F$. Hence, $G$ is solvable and by [14, 14.4.4, p. 440], $G\subseteq F$.
\end{proof}

\begin{lemma}\label{lem:2.3} Let $D$ be a weakly locally finite division ring with center $F$. If $G$ is a (locally solvable)-by-(locally finite) subnormal subgroup of $D^*$ then $G\subseteq F$.
\end{lemma}

\begin{proof} By the assumption, there exists a locally solvable normal subgroup $S$ of $G$ such that $G/S$ is locally finite. Take any $x, y\in S$ and denote by $K$ the division subring of $D$ generated by $x$ and $y$. Then, $K$ is a finite dimensional vector space over its center $Z(K)$ and $S\cap K^*$ is a locally solvable subnormal subgroup of $K^*$. Hence, by Lemma \ref{lem:2.2}, 
$S\cap K^*\subseteq Z(K)$ and, in particular $xy=yx$. So, $S$ is abelian and in view of [14, 14.4.4, p. 440], $S\subseteq F$. Further, as in the proof of Lemma \ref{lem:2.2}, we conclude that $G\subseteq F$.
\end{proof}

Let $D$ be a weakly locally finite division ring with center $F$. Suppose that $Y=\generate{y_1, \ldots, y_s}$ is a finitely generated subgroup of ${\rm GL}_n(D)$. Denote by $K$ the division subring of $D$ generated by all the entries of all matrices $y_1, \ldots, y_s$. Then, $K$ is centrally finite and we denote by $m=[K: Z(K)]$, where $Z(K)$ is  the center of $K$. Now, $Y$ can be viewed as a subgroup of ${\rm GL}_{mn}(Z(K))$. So, $Y$ is a $CZ$-group with  Zariski's topology, defined on $M_{mn}(Z(K))$. It is clear that every subgroup $X$ of $Y$ is a $CZ$-group too. For such a subgroup $X$, denote by $X^0$ the connected component of $X$ containing the identity $1$. Now, for a subgroup $G$  of ${\rm GL}_n(D)$, we denote by $G^+:=\cup_Y(G\cap Y)^0$, where $Y$ runs on the set of all finitely generated subgroups of $G$. 

\begin{lemma}\label{lem:2.4} $G^+$ is a normal subgroup of $G$ and $G/G^+$ is locally finite.
\end{lemma}
\begin{proof} Consider any two elements $x$ and $y$ of $G^+$.  Suppose  $Y_1$ and $Y_2$ are finitely generated subgroups of ${\rm GL}_n(D)$ with finite generating subsets $S_1$ and $S_2$ respectively, such that $x\in (G\cap Y_1)^0$ and $y\in (G\cap Y_2)^0$. Denote by $Y$ the subgroup of ${\rm GL}_n(D)$ generated by $S_1$ and $S_2$. Then, $x, y\in (G\cap Y)^0$, so $xy\in (G\cap Y)^0$ and consequently, $x, y\in G^+$. Evidently, the inverse element of any element from $G^+$ belongs to $G^+$. So, $G^+$ is a subgroup of $G$. Now, suppose that $x\in G^+$ and $g\in G$. Then, $x\in (G\cap Y)^0$ and $g\in (G\cap Y)$ for some  finitely generated subgroup $Y$of $G$. If $S$ is a finite generating subset of $Y$ then we denote by $Y_g$ the subgroup of $G$ generated by $g$ and $S$.   By [17, Lemma 5.2], $gxg^{-1}\in (G\cap Y_g)^0\leq G^+$. Therefore, $G$ is normal in $G$. 

Now, suppose $x_1, \ldots, x_r\in G$. Denote by $H$  the subgroup of $G$ generated by $x_1, \ldots, x_r$. Then, clearly
$$\generate{x_1G^+, \ldots, x_rG^+}=HG^+/G^+.$$

We have $HG^+/G^+\cong H/H\cap G^+$ and the map 
$$\varphi: H/H\cap G^+\longrightarrow H/H^+$$
sending every $h(H\cap G^+)$ to $hH^+$ is injective group homomorphism. By [15, 3.1.3, p.81], $H/H^+$ is finite, so $H/H\cap G^+$ is finite too. Thus, $G/G^+$ is locally finite, as it was required to prove.
\end{proof}

\begin{lemma}\label{lem:2.5} Let $D$ be a weakly locally finite division ring. If $G$ is a subgroup of ${\rm GL}_n(D)$ then the following statements are equivalent:

(i) $G$ is locally solvable-by-finite.

(ii) $G$ is (locally solvable)-by-(locally finite).

(iii) $G^+$ is locally solvable.
\end{lemma}
\begin{proof} $(iii) \Longrightarrow (ii)$: This implication is obvious by Lemma \ref{lem:2.4}.

$(ii)\Longrightarrow (i)$: Suppose $H$ is a finitely generated subgroup of $G$. Then, there exists some normal locally solvable subgroup $A$ of $G$ such that $G/A$ is locally finite. If $B=A\cap H$ then $H/B\cong HA/A$, so $H/B$ is finite. Then, $B$ is finitely generated subgroup of $A$, so it is solvable.

$(i)\Longrightarrow (iii)$:  Suppose  $(i)$ holds. In the first, we prove that for every finitely generated subgroup $Y$ of ${\rm GL}_n(D)$,  the subgroup $X^0$ is solvable, where $X=G\cap Y$.  Thus, let $Y=\generate{y_1, \ldots, y_s}$ be such a subgroup of ${\rm GL}_n(D)$ . Then, $X$ is locally solvable-by-finite. Let $R$ be a subring of $D$ generated by the simple subfield of $D$ and all the entries of matrices $y_1, \ldots, y_s$. Then, $X\leq Y\leq {\rm GL}_n(R)$ and by applying [17, Lemma 10.12, p. 143], $X$ is solvable-by-finite. So, there exists a normal solvable subgroup $A$ of $X$ such that $X/A$ is finite. Denote by $K$ the division subring of $D$ generated by all the entries of all  matrices $y_1, \ldots, y_s$. Since $D$ is weakly locally finite, $K$ is a finite dimensional vector space over its center $Z(K)$. If $m=[K: Z(K)]$ then $Y$ can be viewed as a subgroup of ${\rm GL}_{mn}(Z(K))$. Consider $Y$ as a topological space with induced  Zariski's topology in ${\rm GL}_{mn}(Z(K))$ and denote by $\overline{A}$ the closure of $A$. Then, by [17, Lemma 5.9, Lemma 5.11, p. 78], $\overline{A}$ is a normal solvable subgroup of $X$.  Moreover, $\overline{A}$ is of finite index in $X$. Hence, by [17, Lemma 5.3, p. 75], $X^0\subseteq \overline{A}$ and it follows that $X^0$ is solvable. Now, suppose $H=\generate{x_1, \ldots, x_r}$ is a finitely generated subgroup of $G^+$. Then, for every $1\leq i\leq r$, there exists a finitely generated subgroup $Y_i$ of ${\rm GL}_n(D)$ such that $x_i\in (G\cap Y_i)^0$. Suppose $S_i$ is a finite generating set of $Y_i (1\leq i\leq r$) and denote by $Y_0$ the subgroup of ${\rm GL}_n(D)$ generated by all  $S_i$. If $X=G\cap Y_0$ then $H\leq X^0$ and by what we have proved above, $X^0$ is solvable, so is $H$. Thus, $G^+$ is locally solvable.
\end{proof}

\begin{lemma}\label{lem:2.6} Let $D$ be a weakly locally finite division ring and $G$ be a subgroup of $D^*$. If $G$ contains no non-cyclic free subgroups then $G$ is (locally solvable)-by-(locally finite).
\end{lemma}

\begin{proof} By Theorem \ref{th:2.1}, $G$ is locally solvable-by-finite. Now, by Lemma \ref{lem:2.5}, $G$ is (locally solvable)-by-(locally finite).
\end{proof}

Now, we are ready to get the main result in this paper.

\begin{theorem}\label{th:2.7} Let $D$ be a weakly locally finite division ring with center $F$. If $G$ is a non-central subnormal subgroup of $D^*$, then $G$ contains non-cyclic free subgroups.
\end{theorem}
\begin{proof} Assume that $G$ contains no non-cyclic free subgroups. By Lemma \ref{lem:2.6}, $G$ is (locally solvable)-by-(locally finite). Now, by Lemma \ref{lem:2.3}, $G\subseteq F$, a contradiction.
\end{proof}

\end{document}